\documentclass[12pt]{amsart}

\usepackage{amssymb}
\usepackage{pdfsync}
\usepackage{cite}
\usepackage{slashbox}	
\usepackage{amsfonts}
\usepackage{cancel}	
\usepackage{amstext}
\usepackage{amsthm,amsmath,amsfonts,amssymb,amscd,latexsym,txfonts,stmaryrd}
\usepackage{color}
\usepackage{epsfig}
\usepackage[all,cmtip]{xy}
\usepackage{enumerate}

\newtheorem{theorem}{Theorem}[section]
\newtheorem*{theoremd*}{Definition/Theorem}
\newtheorem*{theorem1*}{Theorem 1}
\newtheorem*{theorem*}{Theorem 1}
\newtheorem*{problem*}{Problem}

\newtheorem*{question*}{Question}

\newtheorem{definition}{Definition}[section]
\newtheorem*{remarks*}{Remarks}
\newtheorem*{claim*}{Claim}

\newtheorem*{remark*}{Remark}
\newtheorem{remark}{Remark}[section]

\newtheorem*{hLt*}{Hard Lefschetz Theorem}
\newtheorem*{HRR*}{Hodge-Riemann Bilinear Relations}
\newtheorem*{basisthm*}{Basis Theorem}
\newtheorem*{relbasisthm*}{Relative Basis Theorem}
\newtheorem*{primdecomp*}{Primitive Decomposition Theorem}
\newtheorem*{PD*}{Poincar\'e Duality}
\newtheorem*{JBC*}{Watanabe's Bold Conjecture}
\newtheorem{proposition}{Proposition}[section]
\newtheorem{lemma}{Lemma}[section]
\newtheorem{corollary}{Corollary}[section]
\newtheorem*{corollary1*}{Corollary 1}
\newtheorem*{corollary*}{Corollary}

\newcommand{\Z}{{\mathbb Z}}

\newcommand{\N}{{\mathbb N}}

\newcommand{\F}{{\mathbb F}}

\numberwithin{equation}{section}

\begin{document}

\title[Bold Conjecture]{Some Remarks On Watanabe's Bold Conjecture}
\author{Chris McDaniel}
\address{Dept. of Math. and Comp. Sci.\\
Endicott College\\
Beverly, MA 01915}
\email{cmcdanie@endicott.edu}






\begin{abstract} 
At the 2015 Workshop on Lefschetz Properties of Artinian Algebras, Junzo Watanabe conjectured that every graded Artinian complete intersection algebra with the standard grading can be embedded into another such algebra cut out by quadratic generators.  We verify this conjecture in the case where the defining polynomials split into linear factors. 
\end{abstract}
\maketitle



\section{Introduction}
Let $\F$ be any field, and let $R=\F[x_1,\ldots,x_n]$ be the polynomial ring in $n$ variables, endowed with the standard grading, i.e. $\deg(x_i)=1$ for each $i$.  A regular sequence is a sequence of homogeneous polynomials $f_1,\ldots,f_n\in R$ such that $f_1\neq 0$ and for each $2\leq i\leq n$, $f_i$ is a non-zero divisor in the quotient $R\left/\langle f_1,\ldots,f_{i-1}\rangle\right.$.  The quotient ring $R\left/\langle f_1,\ldots,f_n\rangle\right.$ is known as a graded Artinian complete intersection algebra with the standard grading, but in this paper we shall just use the term \emph{complete intersection}.  A complete intersection in which the generators $f_1,\ldots,f_n$ are all quadratic forms is called a \emph{quadratic complete intersection}.  Recall that a complete intersection is always Gorenstein, hence its socle is a one dimensional graded subspace.  A graded $\F$ algebra homomorphism between two complete intersection algebras is called an \emph{embedding} if it maps a socle generator onto a socle generator.  Note that a complete intersection can embed into another complete intersection only if the two have the same socle degree.

In their recent paper \cite{HWW}, Harima, Wachi, and Watanabe have shown that in a ``generic'' quadratic complete intersection equipped with an appropriate action of a symmetric group, that the invariant subring under a Young subgroup is again a complete intersection algebra with the standard grading.  At the 2015 Workshop on Lefschetz Properties of Artinian Algebras at University of G\"ottingen, Junzo Watanabe asked to what extent complete intersections arise in this way.  Specifically he posed the following ``rather bold conjecture'' which we shall heretofore refer to as Watanabe's bold conjecture:
\begin{JBC*}
Every complete intersection algebra embeds into some quadratic complete intersection algebra.
\end{JBC*}
In this paper we prove that this conjecture holds for complete intersections cut out by polynomials that split into linear factors.  To wit
\begin{theorem}
\label{thm:SplitJBC}
Suppose that $f_1,\ldots,f_n$ form a regular sequence in $R$, and let $I=\langle f_1,\ldots,f_n\rangle$ be the ideal they generate.  Further suppose that for each $1\leq i\leq n$ the polynomial $f_i$ splits into a product of linear factors, i.e. $f_i=\ell_{i,1},\ldots,\ell_{i,N_i}$ for some linear forms $\ell_{i,j}\in V^*$.  Then the complete intersection $R/I$ satisfies Watanabe's bold conjecture.
\end{theorem}

To prove Theorem \ref{thm:SplitJBC} we first show that a regular sequence as in Theorem \ref{thm:SplitJBC} has a ``normal form''.  Then we find a quadratic complete intersection into which our given complete intersection embeds, in its normal form.

\section{Preliminaries}
\subsection{Regular Sequences}
Let $R=\F[x_1,\ldots,x_n]$ be the polynomial ring in $n$ variables with the standard grading.
\begin{definition}
\label{def:RS}
A sequence of homogeneous polynomials of positive degree $f_1,\ldots,f_k\in R$ is called a \emph{partial regular sequence} if $f_1\neq 0$ and for each $2\leq i\leq k$, $f_i$ is not a zero divisor in the quotient ring $R\left/\langle f_1,\ldots,f_{i-1}\rangle\right.$.  A \emph{regular sequence} is a partial regular sequence of length $n$.
\end{definition}
\begin{remark}
What we refer to as a partial regular sequence here is usually referred to as just a regular sequence.  We use the adjective ``partial'' here since we want all of our regular sequences to have length $n$.
\end{remark}

\begin{lemma}
\label{lem:RegPerm}
If $f_1,\ldots,f_k$ is a partial regular sequence, then $f_{\sigma(1)},\ldots,f_{\sigma(k)}$ is also a partial regular sequence for any permutation $\sigma\in\mathfrak{S}_k$.
\end{lemma}
\begin{proof}
See Matsumura \cite{Mat}[Corollary pg. 127].
\end{proof}

\begin{lemma}
\label{lem:RegArt}
A sequence of homogeneous polynomials of positive degree $f_1,\ldots,f_n$ is a regular sequence if and only if the quotient $R\left/\langle f_1,\ldots,f_n\rangle\right.$ is a finite dimensional vector space over $\F$.
\end{lemma}
\begin{proof}
This follows from the graded analogues of Theorems 14.1 and 17.4 in Matsumura \cite{Mat}.
\end{proof}

\begin{lemma}
\label{lem:RegFact}
Suppose that $f_1,\ldots,f_k$ is a partial regular sequence, and that $f_k=g\cdot h$ for some homogeneous positive degree polynomials $g$ and $h$.  Then $f_1,\ldots,f_{k-1},g$ is also a partial regular sequence.
\end{lemma}
\begin{proof}
With $f_1,\ldots,f_k$, $g$, and $h$ as above, it suffices to show that $g$ is not a zero divisor in the quotient ring $R\left/\langle f_1,\ldots,f_{k-1}\rangle\right.$.  But of course if $g$ cannot be a zero divisor, since otherwise $f_k$ would be one, contradicting the regularity assumption on $f_1,\ldots,f_k$.
\end{proof}

\begin{corollary}
\label{cor:RegFact}
If $f_1,\ldots,f_k$ is a partial regular sequence, and $f_i$ splits into a product of linear factors, i.e. $f_i=\ell_{i,0}\cdots \ell_{i,N_i}$ for each $1\leq i\leq k$, then $\ell_{1,j_1},\ldots,\ell_{k,j_k}$ is also a partial regular sequence for each $1\leq j_i\leq N_i$ $1\leq i\leq k$.
\end{corollary}
\begin{proof}
Start with $f_k=\ell_{k,0}\cdots\ell_{k,N_k}$.  By Lemma \ref{lem:RegFact}, we know that $f_1,\ldots,f_{k-1},\ell_{k,j_k}$ is a partial regular sequence for each $1\leq j_k\leq N_k$.  By Lemma \ref{lem:RegPerm}, we know that $\ell_{k,j_k},f_1,\ldots,f_{k-1}$ is also a partial regular sequence for each $1\leq j_k\leq N_k$.  Now repeat.
\end{proof}

The following proposition gives us a normal form for a split regular sequence as described in the statement of Theorem \ref{thm:SplitJBC}.
\begin{proposition}
\label{prop:SplitNorm}
Let $f_1,\ldots,f_n$ be a regular sequence such that each $f_i$ splits as a product of linear forms, and let $I$ be the ideal generated by them.  Then after a linear change of coordinates, we can write 
$$f_i=x_i\prod_{j=0}^{N_i-1}(x_i-\sum_{j\neq i}\lambda_{i,j}x_j), \ \ \ \text{for some $\lambda_{i,j}\in\F$}.$$
\end{proposition}
\begin{proof}
For each $1\leq i\leq n$ write $f_i=\ell_{i,0}\cdots\ell_{i,N_i}$ for $1\leq i\leq n$.  By Corollary \ref{cor:RegFact}, the linear forms $\ell_{1,N_1},\ldots,\ell_{n,N_n}$ are a regular seqence hence they must be linearly independent, hence $\ell_i\mapsto x_i$ is a linear change of coordinates.  Now we can rewrite $f_i=x_i\cdot\prod_{j=0}^{N_i-1}\hat{\ell}_{i,j}$, where $\hat{\ell}_{i,j}=\sum_{k=1}^na_{i,j}^kx_k$ for some scalars $a_{i,j}^k$.  We would like to know that the coefficient $a_{i,j}^i\neq 0$ for each $0\leq j\leq N_i-1$.  Again by Corollary \ref{cor:RegFact} the sequence $x_1,\ldots,\ell_{i,j},\ldots,x_n$ must be regular for each $0\leq j\leq N_i-1$.  On the other hand, if $a_{i,j}^i=0$ the linear form $\ell_{i,j}$ would be zero in the quotient $R\left/\langle x_1,\ldots,\hat{x_i},\ldots,x_n\rangle\right.$ (where the ``hat'' indicates omission), contradicting regularity.  Therefore $a_{i,j}^i$ must be non-zero.  Finally we can define another change of coordinates $x_i\mapsto\frac{1}{\prod_{j=0}^{N_i-1}a_{i,j}^i}x_i$ to get our desired form.
\end{proof}

We also will make use of the following lemma, lifted from a recent paper of Abedelfatah \cite{Abed}
\begin{lemma}
\label{lem:Abed}
Suppose that $f_1=x_1L_1,\ldots,f_n=x_nL_n$ is a regular sequence in $\F[x_1,\ldots,x_n]$, where $L_i$is a linear form for $1\leq i\leq n$.  Then the quotient 
$$\F[x_1,\ldots,x_n]\left/\langle f_1,\ldots,f_n\rangle\right.$$
is spanned (as an $\F$ vector space) by the equivalence classes of square free monomials from $\F[x_1,\ldots,x_n]$.
\end{lemma}
\begin{proof}
See Lemma 3.1 and Lemma 3.2 in the aforementioned paper \cite{Abed}.
\end{proof}

\begin{corollary}
\label{cor:MEAbed}
With notation as in Lemma \ref{lem:Abed}, a socle generator for the quotient
$$\F[x_1,\ldots,x_n]\left/\langle f_1,\ldots,f_n\rangle\right.$$
is $L_1\cdots L_n$.
\end{corollary}
\begin{proof}
By Corollary \ref{cor:RegFact}, the linear forms $L_1,\ldots,L_n$ must form a regular sequence, hence must be linearly independent over $\F$.  Hence the polynomial rings $\F[x_1,\ldots,x_n]$ and $\F[L_1,\ldots,L_n]$ are identical.  Now applying Lemma \ref{lem:Abed} to the regular sequence $f_1=x_1\cdot L_1,\ldots,f_n=x_n\cdot L_n$, regarding $L_1,\ldots,L_n$ as the coordinate functions, we see that $\F[L_1,\ldots,L_n]\left/\langle f_1,\ldots,f_n\rangle\right.$ is spanned by square free monomials in the $L_1,\ldots,L_n$.  In particular, the socle is therefore generated by the unique square free monomial in degree $n$, namely $L_1\cdots L_n$.
\end{proof}

\subsection{Monomial Orderings}
We can endow the set of monomials in the polynomial ring $\F[x_1,\ldots,x_n]$ with a total ordering by declaring that
$$x^{a_1}_1\cdots x^{a_n}_n< x_1^{b_1}\cdots x_n^{b_n}$$
if $a_1+\cdots+a_n<b_1+\cdots+b_n$ or if $a_1+\cdots+a_n=b_1+\cdots+b_n$ and $a_n=b_n,\ldots,a_{n-j+1}=b_{n-j+1}$ and $a_{n-j}< b_{n-j}$ for some $1\leq j\leq n$.  Note that with this order, we have
$$x_1<\cdots<x_n.$$
This is called the \emph{graded lexicographic} or \emph{grlex} monomial ordering.
\begin{proposition}
\label{prop:Grlex}
The grlex ordering has the following properties.
\begin{enumerate}
\item For any monomials $m_1,m_2,m\in\F[x_1,\ldots,x_n]$ we have
$$m_1\leq m_2 \ \ \Leftrightarrow \ \ \ m\cdot m_1\leq m\cdot m_2$$
\item For any fixed monomial $m\in\F[x_1,\ldots,x_n]$ there are only finitely many monomials less than $m$.
\end{enumerate}
\end{proposition}
\begin{proof}
The proof of (1) is obvious.  To see (2), fix a monomial $m\in\F[x_1,\ldots,x_n]$ and suppose that $m'\leq m$ is any lesser monomial.  Then by definition, $\deg(m')\leq\deg(m)$, hence $m'$ belongs to the finite set 
$\left\{\substack{\text{monomials in $\F[x_1,\ldots,x_n]$}\\ \text{of degree $\leq\deg(m)$}\\}\right\}.$
\end{proof} 

We shall use monomial orderings to argue that certain quotients of polynomial rings are Artinian.

\section{The Main Result}
Fix a field $\F$, fix integers $N_1,\ldots,N_n$, and suppose $f_1,\ldots,f_n$ is a regular sequence in $\F[x_1,\ldots,x_n]$ of degrees $N_1,\ldots,N_n$, respectively, and let $I=\langle f_1,\ldots,f_n\rangle$ be the ideal they generate.  We will further assume that each $f_i$ splits as a product of linear forms, so that we may write the quotient $A\coloneqq \F[x_1,\ldots,x_n]/I$ in its normal form
$$A=\F[x_1,\ldots,x_n]\left/\left\langle \left.x_i\prod_{k=0}^{N_i-1}\left(x_i-\sum_{j\neq i}\lambda_{i,j}^kx_j\right)\right|1\leq i\leq n\right\rangle\right.$$
for some $\lambda_{i,j}^k\in\F$ for $1\leq i\leq n$, $1\leq j\leq N_i$ and $0\leq k\leq N_i-1$, by Proposition \ref{prop:SplitNorm} above.

Now consider the polynomial ring $\F\left[Z_{i,k}\left|1\leq i\leq n, \ 1\leq k\leq N_i\right.\right]$.  For each $1\leq i\leq n$ define the linear form 
$$L_{i,N_i}\coloneqq Z_{i,1}+\cdots+Z_{i,N_i-1}-Z_{i,N_i}+\sum_{j\neq i}\lambda_{i,j}^0Z_{j,N_j}$$
and for $1\leq k\leq N_i-1$ define
$$L_{i,k}\coloneqq Z_{i,1}+\cdots+Z_{i,k}-\left(\sum_{j\neq i}(\lambda_{i,j}^{N_i-k}-\lambda^0_{i,j})Z_{j,N_j}\right).$$

Define the quotient ring 
$$\hat{A}\coloneqq \F\left[Z_{i,j}\left|1\leq i\leq n, \ \ 1\leq j\leq N_i\right.\right]\left/\langle Z_{i,j}\cdot L_{i,j}\left| 1\leq i\leq n, \ \ 1\leq j\leq N_i\right.\rangle\right.$$

\begin{proposition}
\label{prop:WellDef}
There is a well defined $\F$ algebra map $\phi\colon A\rightarrow \hat{A}$ defined by $\phi(x_i)=Z_{i,N_i}$, $1\leq i\leq n$.
\end{proposition}
\begin{proof}
We need to show that for each $1\leq i\leq n$ we have the following relation in $\hat{A}$:
$$Z_{i,N_i}\prod_{k=0}^{N_i-1}\left(Z_{i,N_i}-\sum_{j\neq i}\lambda_{i,j}^kZ_{j,N_j}\right)\equiv 0.$$
Note that in $\hat{A}$ we have the relations
\begin{equation}
\label{eq:HARel1}
Z_{i,N_i}\cdot\left(Z_{i,N_i}-\sum_{j\neq i}\lambda_{i,j}^0Z_{j,N_j}\right)\equiv Z_{i,N_i}\cdot (Z_{i,1}+\cdots+Z_{i,N_{i}-1})
\end{equation}
and for $1\leq k\leq N_i-1$ we have 
\begin{equation}
\label{eq:HARel2}
Z_{i,N_i}\cdot\left(Z_{i,N_i}-\sum_{j\neq i}\lambda_{i,j}^kZ_{j,N_j}\right)\equiv Z_{i,N_i}\cdot\left(L_{i,N_i-k}+Z_{i,N_i-k+1}+\cdots+Z_{i,N_i-1}\right).
\end{equation}
Combining Equivalences \eqref{eq:HARel1} and \eqref{eq:HARel2} we see that 
\begin{equation}
\label{eq:HARel3}
Z_{i,N_i}\cdot\prod_{k=0}^{N_i-1}\left(Z_{i,N_i}-\sum_{j\neq i}\lambda_{i,j}^kZ_{j,N_j}\right)\equiv 
Z_{i,N_i}\left(Z_{i,1}+\cdots+Z_{i,N_i-1}\right)\prod_{k=1}^{N_i-1}\left(L_{i,N_i-k}+Z_{i,N_i-k+1}+\cdots+Z_{i,N_i-1}\right).
\end{equation}
Note that the product on the RHS of Equivalence \eqref{eq:HARel3} collapses, i.e.
\begin{equation}
\label{eq:HARel4}
\prod_{k=1}^{N_i-1}\left(L_{i,N_i-k}+Z_{i,N_i-k+1}+\cdots+Z_{i,N_i-1}\right)\equiv \prod_{k=1}^{N_i-1}L_{i,N_i-k}
\end{equation}
since $L_{i,N_i-1}$ is killed by $Z_{N_i-1}$ and inductively $L_{i,N_i-1}\cdots L_{i,N_i-k}$ is killed by $Z_{i,N_i-k}+\cdots+Z_{i,N_i-1}$.  Plugging Equivalence \eqref{eq:HARel4} into Equivalence \eqref{eq:HARel3} we finally get the relation
\begin{align*}
Z_{i,N_i}\cdot\prod_{k=0}^{N_i-1}\left(Z_{i,N_i}-\sum_{j\neq i}\lambda_{i,j}^kZ_{j,N_j}\right)\equiv & 
Z_{i,N_i}\left(Z_{i,1}+\cdots+Z_{i,N_i-1}\right)\prod_{k=1}^{N_i-1}L_{i,N_i-k}\\
\equiv &0
\end{align*}
where the last equivalence because $Z_{i,1}+\cdots+Z_{i,N_i-1}$ kills the product $\prod_{k=1}^{N_i-1}L_{i,N_i-k}$ in $\hat{A}$.  
\end{proof}

\begin{corollary}
\label{cor:HACI}
If $A$ is Artinian, then $\hat{A}$ is also Artinian.
\end{corollary}
\begin{proof}
Suppose that $A$ is Artinian.  Rewriting the defining relations for $\hat{A}$ we get that for each $1\leq i\leq n$
\begin{align*}
Z_{i,N_i}^2\equiv & Z_{i,N_i}\left(Z_{i,1}+\cdots+Z_{i,N_i-1}+\sum_{j\neq i}\lambda_{i,j}^0Z_{j,N_j}\right)\\
Z_{i,N_i-1}^2\equiv & Z_{i,N_{i}-1}\left(\sum_{j\neq i}\left(\lambda_{i,j}^1-\lambda_{i,j}^0\right)Z_{j,N_j}-\left(Z_{i,1}+\cdots+Z_{i,N_i-2}\right)\right)\\
\vdots &\\
Z_{i,N_i-k}^2\equiv & Z_{i,N_i-k}\left(\sum_{j\neq i}\left(\lambda_{i,j}^k-\lambda_{i,j}^0\right)Z_{j,N_j}-\left(Z_{i,1}+\cdots+Z_{i,N_i-k-1}\right)\right)\\
\vdots & \\
Z_{i,1}^2\equiv & Z_{i,1}\left(\sum_{j\neq i}\left(\lambda_{i,j}^{N_i-1}-\lambda_{i,j}^0\right)Z_{j,N_j}\right)
\end{align*}
Now endow the set of monomials in $\F[Z_{i,j}\left|1\leq i\leq n, \ \ 1\leq j\leq N_i\right.]$ with the grlex ordering stemming from the following ordering of the variables:
$$Z_{1,N_1}\prec\cdots\prec Z_{n,N_n}\prec Z_{1,1}\prec\cdots\prec Z_{1,N_1-1}\prec\cdots\prec Z_{n,1}\prec\cdots\prec Z_{n,N_n-1}.$$
Then we see that by using the above relations, we can express the squares of each of the variables, with the exception of $Z_{1,N_1}, \ldots,Z_{n,N_n}$, as monomials of strictly lesser order.  Therefore, by virtue of Proposition \ref{prop:Grlex}, every monomial in $\hat{A}$ is equivalent to an $\F$ linear combination of monomials of the form 
$$Z_{1,N_1}^{A_1}\cdots Z_{n,N_n}^{A_n}\prod_{i=1}^{n}\prod_{j=1}^{N_i-1}Z_{i,j}^{\epsilon_{i,j}}, \ \ A_i\in\N, \ \ \epsilon_{i,j}\in\left\{0,1\right\}.$$
Hence to show that $\hat{A}$ is Artinian, i.e. a finite dimensional $\F$ vector space, it suffices to show that the elements $Z_{j,N_j}$ are nilpotent in $\hat{A}$.  But $Z_{j,N_j}=\phi(x_j)$ and if $A$ is Artinian, $x_j$ is nilpotent for all $j$, and the result now follows.
\end{proof}  

Now we see that if $A$ is Artinian, then so is $\hat{A}$, and hence by Lemma \ref{lem:RegArt}, the sequence $Z_{i,j}\cdot L_{i,j}$, $1\leq i\leq n$, $1\leq j\leq N_i$ must be a regular sequence in $\F[Z_{i,j}\left| 1\leq i\leq n, \ 1\leq j\leq N_i\right.]$.  Hence in this case, Corollary \ref{cor:MEAbed} tells us that a socle generator for $\hat{A}$ is given by $\prod_{i=1}^n\prod_{j=1}^{N_i}L_{i,j}$.   
\begin{proposition}
\label{prop:Main}
If $A$ is Artinian, then so is $\hat{A}$ and the map $\phi\colon A\rightarrow \hat{A}$ defined above is an embedding.
\end{proposition}
\begin{proof}
By the preceding discussion, we need only show that $\prod_{i=1}^n\prod_{j=1}^{N_i}L_{i,j}$ is in the image of $\phi$.  The claim is that we have the following equivalence:
$$\prod_{i=1}^n\prod_{k=0}^{N_i-1}\left(Z_{i,N_i}-\sum_{j\neq i}\lambda_{i,j}^kZ_{j,N_j}\right)\equiv C\cdot \prod_{i=1}^n\prod_{j=1}^{N_i}L_{i,j} \ \ \text{for some $C\in\F^\times$}.$$
Then since the LHS is clearly in the image of $\phi$, this will prove the desired result.  For a shorthand notation, we write 
$$M_i\coloneqq \prod_{k=0}^{N_i-1}\left(Z_{i,N_i}-\sum_{j\neq i}\lambda_{i,j}^kZ_{j,N_j}\right)$$
and we shall write 
$$L_i\coloneqq \prod_{k=0}^{N_i-1}L_{i,N_i-k}.$$
In this notation we want to show that in $\hat{A}$ we have the relation
$$\prod_{i=1}^{n}M_i\equiv C\cdot \prod_{i=1}^nL_i \ \ \text{for some $C\in\F^\times$}.$$
In fact we will show that for each $1\leq i\leq n$ we always have the equivalence in $\hat{A}$ 
$$M_1\cdots M_n\equiv (-1)^i\cdot L_1\cdots L_i\cdot M_{i+1}\cdots M_n.$$
We prove this by induction on $i$.  The base case $i=1$ asserts that we have the equivalence in $\hat{A}$
$$M_1\cdots M_n\equiv -L_1\cdot M_2\cdots M_n.$$
To see this, we first recall some computations performed previously.  To start recall that we had for each $1\leq i\leq n$ and for each $1\leq k\leq N_i-1$ we have
\begin{equation}
\label{eq:Soc1}
Z_{i,N_i}\cdot(Z_{i,N_i}-\sum_{j\neq i}\lambda_{i,j}^kZ_{j,N_j})\equiv Z_{i,N_i}\cdot(L_{i,N_i-k}+Z_{i,N_i-k+1}+\cdot+Z_{i,N_i-1}).
\end{equation}
Also since the sum $Z_{i,N_i-j}+\cdots+Z_{i,N_i-1}$ annihilates the product $\prod_{k=1}^jL_{i,N_i-k}$, the entire product collapses, i.e.
\begin{equation}
\label{eq:Soc2}
\prod_{k=1}^{N_i-1}(L_{i,N_i-k}+Z_{i,N_i-k+1}+\cdots+Z_{i,N_i-1})\equiv \prod_{k=1}^{N_i-1}L_{i,N_i-k},
\end{equation}
for each $1\leq i\leq n$.  Finally we recall the equation
\begin{equation}
\label{eq:Soc3}
Z_{i,N_i}=Z_{i,1}+\cdots+Z_{i,N_i-1}-L_{i,N_i}+\sum_{j\neq i}\lambda^0_{i,j}Z_{j,N_j}.
\end{equation}
Now combining Equivalences \eqref{eq:Soc1}, \eqref{eq:Soc2} and Equation \eqref{eq:Soc3}, we obtain the following equivalence:
\begin{align}
\label{eq:Soc4}
\nonumber M_i= & \left(Z_{i,N_i}-\sum_{j\neq i}\lambda_{i,j}^0Z_{j,N_j}\right)\prod_{k=1}^{N_i-1}\left(Z_{i,N_i}-\sum_{j\neq i}\lambda_{i,j}^k\Z_{j,N_j}\right)\\
\nonumber= & Z_{i,N_i}\prod_{k=1}^{N_i-1}\left(Z_{i,N_i}-\sum_{j\neq i}\lambda_{i,j}^kZ_{i,j}^k\right)\\
\nonumber&-\left(\sum_{j\neq i}\lambda_{i,j}^kZ_{j,N_j}\right)\prod_{k=1}^{N_i-1}\left(Z_i-\sum_{j\neq i}\lambda_{i,j}^kZ_{j,N_j}\right)\\
\nonumber\equiv & Z_{i,N_i}\prod_{k=1}^{N_i-1}\left(L_{i,N_i-k}+Z_{i,N_i-k+1}+\cdots+Z_{i,N_i}\right)\\
\nonumber&-\left(\sum_{j\neq i}\lambda_{i,j}^kZ_{j,N_j}\right)\prod_{k=1}^{N_i-1}\left(Z_i-\sum_{j\neq i}\lambda_{i,j}^kZ_{j,N_j}\right)\\
\nonumber\equiv & Z_{i,N_i}\prod_{k=1}^{N_i-1}L_{i,N_i-k}\\
\nonumber&-\left(\sum_{j\neq i}\lambda_{i,j}^kZ_{j,N_j}\right)\prod_{k=1}^{N_i-1}\left(Z_i-\sum_{j\neq i}\lambda_{i,j}^kZ_{j,N_j}\right)\\
\nonumber\equiv & \left(Z_{i,1}+\cdots+Z_{i,N_i-1}-L_{i,N_i}+\sum_{j\neq i}\lambda_{i,j}^0Z_{j,N_j}\right)\prod_{k=1}^{N_i-1}L_{i,N_i-k}\\
\nonumber&-\left(\sum_{j\neq i}\lambda_{i,j}^kZ_{j,N_j}\right)\prod_{k=1}^{N_i-1}\left(Z_{i,N_i}-\sum_{j\neq i}\lambda_{i,j}^kZ_{j,N_j}\right)\\
\nonumber\equiv & -\prod_{k=0}^{N_i-1}L_{i,N_i-k}\\
\nonumber&-\left(\sum_{j\neq i}\lambda_{i,j}^kZ_{j,N_j}\right)\left(\prod_{k=1}^{N_i-1}\left(Z_{i,N_i}-\sum_{j\neq i}\lambda_{i,j}^kZ_{j,N_j}\right)-\prod_{k=1}^{N_i-1}L_{i,N_i-k}\right)\\
M_i\equiv  & -L_i-\left(\sum_{j\neq i}\lambda_{i,j}^kZ_{j,N_j}\right)\left(\prod_{k=1}^{N_i-1}\left(Z_{i,N_i}-\sum_{j\neq i}\lambda_{i,j}^kZ_{j,N_j}\right)-\prod_{k=1}^{N_i-1}L_{i,N_i-k}\right)
\end{align}
Plugging Equivalence \eqref{eq:Soc4} into the product $M_1\cdots M_n$ we get
$$\left(-L_1-\left(\sum_{j\neq 1}\lambda_{1,j}^kZ_{j,N_j}\right)\left(\prod_{k=1}^{N_1-1}\left(Z_{1,N_1}-\sum_{j\neq 1}\lambda_{i,j}^kZ_{j,N_j}\right)-\prod_{k=1}^{N_i-1}L_{i,N_i-k}\right)\right)\cdot M_2\cdots M_n.$$
But since $Z_{j,N_j}$ annihilates $M_2\cdots M_n$ for all $j\neq 1$, we get the equivalence
$$M_1\cdots M_n\equiv -L_1\cdot M_2\cdots M_n$$
which establishes the base case.  Inductively assume that we have the equivalence
$$M_1\cdots M_n\equiv (-1)^{i-1}\cdot L_1\cdots L_{i-1}\cdot M_{i}\cdots M_n$$ 
for some $i>1$.  Now we substitute Equivalence \eqref{eq:Soc4} into $L_1\cdots L_{i-1}\cdot M_i\cdots M_{n}$ to get
$$(-1)^{i-1}L_1\cdots L_{i-1}\left(-L_i-\left(\sum_{j\neq i}\lambda_{i,j}^kZ_{j,N_j}\right)\left(\prod_{k=1}^{N_i-1}\left(Z_{i,N_i}-\sum_{j\neq i}\lambda_{i,j}^kZ_{j,N_j}\right)-\prod_{k=1}^{N_i-1}L_{i,N_i-k}\right)\right)\cdot M_{i+1}\cdots M_n.$$
Note again that for each $j\neq i$, $Z_{j,N_j}$ annihilates the product $L_1\cdots L_{i-1}\cdot M_{i+1}\cdots M_n$, and so we again obtain
$$M_1\cdots M_n\equiv -(-1)^{i-1}L_1\cdots L_i\cdot M_{i+1}\cdots M_n$$
which completes the induction and proves the claim.  Therefore we have shown that in $\hat{A}$ we have the equivalence
$$M_1\cdots M_n\equiv (-1)^nL_1\cdots L_n$$
and therefore $\phi\colon A\rightarrow\hat{A}$ must be an embedding, as desired.
\end{proof}

\bibliographystyle{plain}
\bibliography{Smith}

\end{document}